\documentclass{gtpart}     % Basic GT/GTM/AGT style
\usepackage{amssymb}
\title[Addendum to ``Commensurations of $F$'']{Addendum to ``Commensurations and Subgroups of
Finite Index of Thompson's Group $F$''}

\author{Jos\'e Burillo}
\givenname{Jos\'e}
\surname{Burillo}
\address{Departament de Matem\`atica Aplicada IV\\
Universitat Polit\`ecnica de Catalunya\\
Escola Polit\`ecnica Superior de Castelldefels\\
Esteve Terrades 5\\
08860 Castelldefels (Barcelona)\\
Spain}
\email{burillo@mat.upc.es}
\urladdr{http://www-ma4.upc.edu/~burillo/}
\author{Sean Cleary}
\givenname{Sean}
\surname{Cleary}
\address{Department of Mathematics R8133\\
The City College of New York\\
Convent Ave \& 138th\\
New York, NY 10031}
\email{cleary@sci.ccny.cuny.edu}
\urladdr{http://www.sci.ccny.cuny.edu/~cleary/}
\author{Claas E. R\"over}
\givenname{Claas}
\surname{R\"over}
\address{School of Mathematics, Statistics and Applied Mathematics\\
National University of Ireland Galway\\
University Road\\
Galway\\
Ireland}
\email{claas.roever@nuigalway.ie}
\urladdr{http://www.maths.nuigalway.ie/~chew/}

\arxivreference{}
\arxivpassword{}

\volumenumber{}
\issuenumber{}
\publicationyear{}
\papernumber{}
\startpage{}
\endpage{}
\doi{}
\MR{}
\Zbl{}
\received{}
\revised{}
\accepted{}
\published{}
\publishedonline{}
\proposed{}
\seconded{}
\corresponding{}
\editor{}
\version{}

\newtheorem{thm}{Theorem}%[section]    % Standard theorem environment
\newtheorem{lem}[thm]{Lemma}          % Lemma environment with numbering 
\theoremstyle{definition}
\newcommand{\nn}{\ensuremath{{\mathrm{\sf I}}\!{\mathrm{\sf N}}}}%
\newcommand{\zz}{\ensuremath{\mathrm{\sf Z}\!\!\!\!\:{\mathrm{\sf Z}}}}%
\newcommand{\qq}{\ensuremath{\mathrm{\sf Q}\!\!\!\!\:\!{\sf I}\,\,}}%
\newcommand{\rr}{\ensuremath{{\mathrm{\sf I}}\!{\mathrm{\sf R}}}}%

\newcommand{\tr}[1]{\mbox{\ensuremath{\textrm{#1}}}}%shorthand for \textrm
\newcommand{\com}[1]{\mbox{$\tr{\rm Com}(#1)$}}%gives Com(?)
\newcommand{\coplus}[1]{\mbox{$\tr{\rm Com}^+(#1)$}}
\newcommand{\plr}{\ensuremath{P}}% PL maps with lots of conditions
\newcommand{\plpr}{\ensuremath{P_+}}% orientation preserving such maps

\newcommand{\quot}[2]{{\left.\raisebox{.2em}{$#1$}\middle/\raisebox{-.2em}{$#2$}\right.}}

\def\autpf{{\rm Aut}^+(F)}

\begin{document}

\begin{abstract}    % type your abstract below
We show that the abstract commensurator of $F$ is composed of four building blocks: two isomorphism types of simple groups, the multiplicative group of the positive rationals and a cyclic group of order two. The main result establishes the simplicity of a certain group of piecewise linear homeomorphisms of the real line.
\end{abstract}

\maketitle

%%%%%%%%%%%%%%%%%%%%   Start of main body of article

The purpose of this note is to extend earlier work \cite{comf}, where we described the commensurator group of Thompson's group $F$. We prove that an interesting subgroup of \com{F} is simple and describe the algebraic structure of \com{F} in terms of short exact sequence of simple groups and the multiplicative group of the positive rationals. For all the details and notation, see the paper \cite{comf}.

\section{The group of eventually periodic maps}
Previously \cite{comf} we described  the commensurator group of $F$ as the group of the eventually integrally periodically affine maps in $\plr$, which is defined in Section 1 of \cite{comf}. These elements may preserve or reverse the orientation of the real line. We also showed that the index-two subgroup $\coplus{F}$ of orientation-preserving maps fits into the short exact sequence
$$
1\longrightarrow K\longrightarrow \coplus F \longrightarrow \qq^*\times\qq^* \longrightarrow 1
$$
whose kernel $K$ is exactly those elements $f$ of $\plpr$ for which there exists $M>0$, and two positive integers $p,p'$ such that
\begin{gather*}
f(t+p)=f(t)+p \tr{ for } t\ge M\tr{ and}\\
f(t+p')=f(t)+p'\tr{ for } t\le -M.
\end{gather*}
Now we can associate to each element $f\in K$ two integrally periodically affine maps $f_+$ and $f_-$, which coincide with $f$ near $\infty$ and $-\infty$, respectively. This property leads to the following definitions.

For $p\in\nn$, we denote by $H_p$ the subgroup of $\plpr$ of $p$-periodically affine maps, that is 
$$ 
H_p=\{f\in\plpr\mid f(t+p)=f(t)+p\textrm{ for all } t\in\rr \}.
$$
Clearly, if $p|q$, then $H_p\subset H_q$, whence we define the subgroup $H$ as a direct limit under inclusion by
$$
H=\bigcup_{p=1}^{\infty}H_p.
$$

The maps $f_+$ and $f_-$ now give rise to a homomorphism
$$
\rho\co K\longrightarrow H\times H,
$$
given by $\rho(f)=(f_-,f_+)$. The kernel consists of the eventually trivial elements, and therefore equals $F'$, the commutator subgroup of $F$ (see \cite{brin} or \cite{comf}). In other words, we get the short exact sequence
$$
1\longrightarrow F'\longrightarrow K \longrightarrow H\times H \longrightarrow 1.
$$

Brin \cite{brin} showed that $\autpf=\rho^{-1}(H_1\times H_1)$ and established the short exact sequence 
$$
1\longrightarrow F\longrightarrow \autpf \longrightarrow T\times T \longrightarrow 1,
$$
where $T$ is Thompson's group $T$ (see \cite{cfp}).  Since we clearly have a map $H_1\rightarrow T$, due to the fact that a map which is $1$-periodically affine can be viewed as a map on the circle $S^1$ given by $\quot{\rr}{\zz}$, an alternative version of this sequence is
$$
1\longrightarrow F'\longrightarrow \autpf \longrightarrow H_1\times H_1 \longrightarrow 1.
$$
These two sequences are related by the short exact sequence
$$
1\longrightarrow A_1\longrightarrow H_1 \longrightarrow T \longrightarrow 1,
$$
whose kernel $A_1$ is the maps $t\mapsto t+k$ for integers $k$. Clearly $A_1$ is isomorphic to $\zz$. 

It is straightforward to verify that any element $\alpha$ of \coplus{F}\ which satisfies $\alpha(t+1)=\alpha(t)+p$ for all $t\in \rr$ conjugates $H_1$ to $H_p$ and $A_1$ to $A_p$, the group of maps of the form $t\mapsto t+kp$ with $k\in\zz$. So we clearly have a short exact sequence
$$
1\longrightarrow A_p\longrightarrow H_p \longrightarrow T \longrightarrow 1.
$$
We note that this extension is, in fact, central, and that one may view this copy of $T$ as acting on the circle of length $p$ given by $\quot{\rr}{p\zz}$. We summarise this discussion as follows.

\begin{thm} The structure of the group $\com F$ and its index-two subgroup $\coplus F$ is given by the following short exact sequences and equalities.
$$
\begin{array}c
1\longrightarrow \coplus F \longrightarrow \com F \longrightarrow C_2 \longrightarrow 1\\
\\
1\longrightarrow K\longrightarrow \coplus F \longrightarrow \qq^*\times \qq^* \longrightarrow 1\\
\\
1\longrightarrow F'\longrightarrow K \longrightarrow H\times H \longrightarrow 1,\qquad H=\displaystyle{\bigcup_{p=1}^{\infty}H_p}\\
\\
1\longrightarrow A_p\longrightarrow H_p \longrightarrow T \longrightarrow 1\\
\\
A_p\cong\zz\text{ is central in }H_p
\\
\end{array}
$$
\end{thm}

\section{Simplicity of the group $H$}

Here we exploit the well-known fact that $T$ is simple (eg. \cite{cfp}) to prove our main result.

\begin{thm} \label{main} The group $H=\{f\in\plpr\mid f(t+p)=f(t)+p\textrm{ for some } p\in \nn \}$ is simple.
\end{thm}

Note that for $p,q\in \nn$ with $p|q$, we have $H_p\subset H_q$ and $A_p\supset A_q$. 
So the theorem says that in the union $H$ the groups $A_p$ cease to be normal. This is due to the 
following.

\begin{lem} 
A normal subgroup of $H_p$ is either $H_p$ or it is contained in $A_p$. \label{lemsimple}
\end{lem}
\begin{proof} In the light of the isomorphism between $H_p$ and $H_1$ which carries $A_p$ to $A_1$, it suffices to consider the case $p=1$. Let $N$ be a normal subgroup of $H_1$ and consider its image in $T$. Since $T$ is simple, the image of $N$ is either $\{1\}$ or the whole $T$. If the image is $\{1\}$, then $N\subset A_1$. So we assume that the image is $T$, which yields the exact sequence
$$
1\longrightarrow B\longrightarrow N \longrightarrow T \longrightarrow 1
$$
with kernel $B=N\cap A_1\subset A_1$. It follows that $B=A_{r}$ for some $r$, and we find that
$$
\quot{H_1}{N}\cong \quot{A_1}{A_{r}}\cong\quot{\zz}{r\zz}.
$$
In particular $\quot{H_1}{N}$ is abelian. The proof will be complete once we show that $H_1$ is equal to its commutator subgroup, because then $N=H_1$. In order to establish this, we recall from \cite{cfp} that $T$ is generated by three elements $x_0$, $x_1$ and $c$ subject to the relators

\[[x_0x_1^{-1},x_0^{-1}x_1x_0],\quad [x_0x_1^{-1},x_0^{-2}x_1x_0^2],\quad x_1x_0^{-1}cx_1c^{-1},\]
\[(x_0^{-1}cx_1)^2x_0^{-1}c^{-1},\quad x_1x_0^{-2}cx_1^2x_0^{-1}x_1^{-1}x_0x_1^{-1}c^{-1}x_0\quad\mathrm{and}\quad c^3.\]

This easily gives rise to a finite presentation for $H_1$ with three generators $x_0$, $x_1$ and $c$ subject to the same relators, except for $c^3$ which has to be replaced by the two relators $[c^3,x_0]$ and $[c^3,x_1]$.
Here $x_0$, $x_1$ and $c$ are the preimages of the corresponding generators for $T$, as defined in \cite{cfp}, with $x_0(0)=x_1(0)=0$ and $c(0)=-1/4$; composition is then to be read from right to left, as in \cite{cfp}. In this case $c^3$ is the map $t\mapsto t-1$ which generates $A_1$.
Modulo the commutator subgroup of $H_1$, the third, fourth and fifth relators yield the relators $x_0^{-1}x_1^2$, $x_0^{-3}x_1^2c$ and $x_0^{-1}x_1$, respectively, which in turn imply $x_0=x_1=c=1$. This proves that $[H_1,H_1] = H_1$.
\end{proof}

\begin{proof}[Proof of Theorem \ref{main}.] Let $N$ be a non-trivial normal subgroup of $H$. According to Lemma \ref{lemsimple}, for each $p$, we have that $N\cap H_p$ is either $H_p$ or it is contained in $A_p$.

We claim that if $N\cap H_p=H_p$ for some $p$, then this happens for all $p\in\nn$. We take $q\in\nn$. Then $N\cap H_{pq}$ is a normal subgroup of $H_{pq}$, and
$$
N\cap H_{pq}\supset N\cap H_p=H_p\supsetneq A_p\supset A_{pq},
$$
which shows that $N\cap H_{pq}= H_{pq}$, by the lemma. Thus, in this case $N$ contains all $H_q$, and hence $N=H$.

The only case left now is that $N\cap H_p\subset A_p$ for all $p$. Since $N$ is non-trivial and the  $A_p$ are infinite cyclic, there exists a $p$ with $N\cap H_p=A_{rp}$ for some $r\ge 1$. But then 
$$
A_{2pr}\supset N\cap H_{2pr}\supset N\cap H_{p}=A_{pr}\supsetneq A_{2pr}
$$
which is a contradiction. Thus the only normal subgroups of $H$ are $H$ and the identity as claimed.
\end{proof}

{We would like to thank the de Br\'un Centre at NUI Galway for its generous support. The first author acknowledges support from MEC grant
MTM2011-25955 and the second author is grateful for funding from NSF \#0811002.}

%%%%%%%%%%%%%%%%%%%%   End of main body of article
%
%                             References
%
%   BiBTeX users uncomment the following line:
%
\bibliographystyle{gtart}

\end{document}